\documentclass[pdflatex,sn-mathphys-num]{sn-jnl}

\usepackage{graphicx}%
\usepackage{multirow}%
\usepackage{amsmath,amssymb,amsfonts,breqn}%
\usepackage{amsthm}%
\usepackage{mathrsfs}%
\usepackage[title]{appendix}%
\usepackage{xcolor}%
\usepackage{textcomp}%
\usepackage{manyfoot}%
\usepackage{booktabs}%
\usepackage{algorithm}%
\usepackage{algorithmicx}%
\usepackage{algpseudocode}%
\usepackage{listings}%
\newtheorem{corollary}{Corollary}[section]

\theoremstyle{thmstyleone}%
\newtheorem{theorem}{Theorem}
\newtheorem{proposition}[theorem]{Proposition}%

\theoremstyle{thmstyletwo}%
\newtheorem{remark}{Remark}%

\theoremstyle{thmstylethree}%
\newtheorem{definition}{Definition}%

\begin{document}

\title[Article Title]{Finite Volume Einstein Finsler Warped Product Manifolds of Non-positive or Non-negative Scalar Curvature}


\author[2]{\fnm{Mohammad} \sur{Aqib}}\email{mohammadaqib@hri.res.in, aqibm449@gmail.com}

\author*[1]{\fnm{Hemangi} \sur{Madhusudan Shah}}\email{hemangimshah@hri.res.in}
\equalcont{These authors contributed equally to this work.}

\author[3]{\fnm{Pankaj} \sur{Kumar}}\email{pankaj.k@galgotiasuniversity.edu.in}
\equalcont{These authors contributed equally to this work.}

\author[4]{\fnm{Anjali} \sur{Shriwastawa}}\email{anjalishrivastav@bhu.ac.in}
\equalcont{These authors contributed equally to this work.}

\affil*[1,2]{Harish-Chandra Research Institute, A CI of Homi Bhaba National Institute, Chhatnag Road, Jhunsi, Prayagraj-211019, India}

\affil[3]{Department of Mathematics, school of Basic Sciences, Galgotias University, Greater Noida, India}

\affil[4]{DST-CIMS, Banaras Hindu University, Varanasi-221005, India}


\abstract{We partially answer a question by A. L. Besse in the Finslerian settings by investigating the obstructions to the existence of Einstein  warped product manifold, when the base is  
  compact or finite volume complete Riemannian manifold 
 and the fiber is a weakly Berwald Finsler manifold. In case of compact base, we establish that, this Finslerian warped must have dimension greater than $3$, if the Ricci scalar is non-positive or non-negative. We relate various conditions on the scalar curvatures of the base, fiber, and the warped product manifold which imply that the resulting warped product is trivial. We also show that, when Ricci scalar of the warped product is positive, then Ricci scalar of the fiber must be positive, in order that the warped product is non-trivial. Moreover, we prove that, if the base manifold is a complete, Riemannian manifold of finite volume with a bounded warping function, and the Ricci scalar of the warped product is non-positive, then the warped product is trivial. Our findings extend the results presented in Kim and Kim paper entitled, "Compact Einstein warped product spaces with non-positive scalar curvature," published in (2003), and Dan Dumitru's results on
 "On compact Einstein warped products," published in (2011),
 in the Finslerian settings. }




\maketitle

\textbf{AMS classification 2020:}~53B40, 52B60.\\
\tableofcontents

\section{\textbf{Introduction}}\label{Intro}
	The notion of warped product plays an important role in Riemannian geometry,
	moreover in geodesic metric spaces.  The warped product was first introduced by Bishop and O'Neill to study Riemannian manifolds of negative curvature (see \cite{lota}). Warped products have been mainly used to construct  new examples of Riemannian manifolds with prescribed 
	curvature conditions.  This
	construction can be extended for Finslerian metrics with some minor restrictions.
	This is motivated by Asanov's papers \cite{gsa,gsa1}, where some models of relativity theory
	are described through the warped product of Finsler metrics. These metrics are in the form of $(\alpha,\beta)$-metrics, which are the generalization of the Randers metrics; which are being asymmetric Finsler metrics in  four-dimensional space-time. The product was later extended to the 
	warped product case of Finsler manifolds by the work of Kozma, Peter and Varga (see \cite{LPC}). \\
 

This article is devoted to the study of Einstein Finslerian 
warped product manifolds, with base a Riemannian manifold
and fiber a weakly Berwald manifold. It turns out that,
in the case of Finsler manifolds of dimensions greater than $2$, in general,  Schur's Lemma  does not hold true. 
Hence, in this case, the Ricci curvature turns out to be a
scalar function on the manifold (see Definition \ref{Einstein} and description after it). 
This is how Finsler geometry differs from Riemannian one drastically. \\

In his inquiry, A.L. Besse (\cite{bas}, p. $265$)  posed the question: "Does  there exist a compact Einstein warped product with nonconstant warping function?" 
This was answered partially and negatively by 
Kim and Kim who proved in \cite{kim}:
\vspace{0.15in}

\begin{theorem}[\textbf{Theorem 1 \cite[\boldmath p. $2573$]{kim}}]\label{kk}
If  $(M, g)$ is an Einstein Riemannian warped product manifold with non-positive scalar curvature and compact base, then $M$ is simply a Riemannian product space.
 \end{theorem}
 
 \vspace{0.15in}
 
Also towards the aforementioned question by Besse,
Dan Dumitru \cite{dam} showed that the fiber of a compact Einstein warped product should have strictly positive scalar curvature, in order that the warped product is non-trivial. And furthermore, he explored the new obstructions to the existence of such spaces
(see Theorem $2.5$ of \cite{dam} for more details).
In particular, it is also observed that a compact Einstein warped product must have a dimension greater than $5$
\cite{dam}.\\

We can extend the Besse question in the Finslerian settings as follows:\\
{\bf Question by A.L. Besse in the Finslerian settings:}
Does  there exist a finite volume Einstein Finsler warped product space with a nonconstant warping function? \\

In this article, we further provide partial negative   answers to Besse's question in the Finslerian settings by demonstrating that, if  $M$ is an Einstein Finslerian warped product of constant non-positive Ricci 
(non-positive Ricci scalar),
 having base a finite  volume complete Riemannian manifold
 (compact Riemannian manifold) and fiber a weakly Berwald Finsler manifold, then 
the resulting manifold $M$ is a weakly Berwald
Finslerian Riemannian product  manifold. In particular, we
 also obtain the Finslerian generalization of  Theorem
 \ref{kk}. See our Theorem \ref{main} and Theorem \ref{mainf}
 proved in \S\ref{sec5}.\\ 
We also answer Besse's question partially negatively,
by exploring obstructions in terms of scalar curvatures 
of base (a compact manifold) and fiber manifold, to the existence of an Einstein Finslerian warped product of positive Ricci scalar function. Thus we also 
extend Theorem $2.5$ of \cite{dam}. See Theorem \ref{thm4.2}
of \S\ref{sec6}.  \\
We also observe that, the dimension of the Einstein Finslerian warped product space with compact base, in case when Ricci scalar 
function is positive or negative, must be at least $4$
(Proposition \ref{prop5.1} of  \S\ref{sec5} ). \\

It is well known that, the Finsler manifolds have 
several non-equivalent volume measures.
To study Einstein Finslerian warped product of 
finite volume, we first establish  their admissible class of 
warped product metrics, by affirming that the warped product
of Finsler manifolds of finite Holmes-Thompson volume,
maximum volume or minimum volume, must be of finite volume
in the same category (Theorem \ref{ht} Theorem \ref{max} of \S\ref{sec3}).\\

The article is divided into $6$ sections. In Section \ref{prelim},
we discuss the basic notions required to prove our main results.
In Section \ref{sec3}, we discuss the formulation of Einstein Finsler 
warped product manifolds and the admissible class of such manifolds
of finite well-known volumes (as mentioned above). 
In Section \ref{sec4}, we show that the warped product of 
two weakly Berwald Finsler manifolds is again weakly Berwald, if the warping function is constant (Corollary \ref{cor4.2}). Section \ref{sec5} and Section \ref{sec6}, respectively, deal with the investigation of Einstein Finslerian warped product of non-positive  and non-negative scalar curvatures, respectively.

	\section{\textbf{Preliminaries}}\label{prelim}
	In this section, we provide some basic notions required to prove our main result.
	Let $M$ be an $n$-dimensional smooth manifold and let $TM:= {\sqcup _{x \in M}}T_xM$ denote the tangent bundle of $M$. Let $TM^0:=TM\backslash\{0\}$
	be the slit tangent bundle of $M$
 and correspondingly we also express
 $T_{x} M^0:=T_{x} M \backslash\{0\}$ as slit tangent space
 at $x \in M$. \\

We first begin with the definition of a Finsler metric.

	\vspace{0.15in}

	\begin{definition}[{Finsler metric \cite[\boldmath p. $12$]{SSZ}}]
		\textnormal{A {\it Finsler metric} on $M$ is a  continuous function $F:TM \to [0,\infty)$ satisfying the following conditions:
			\begin{itemize}
				\item[(a)] $F$ is smooth on $TM^{0}$, 
				\item[(b)] $F$ is a positively 1-homogeneous on the fibers of tangent bundle $TM$,
				\item[(c)] The Hessian of $\frac{1}{2}F^2$ with element $g_{ij}=\frac{1}{2}\frac{\partial ^2F^2}{\partial y^i \partial y^j}$ is positive definite on $TM^0$.
			\end{itemize}
			The pair $(M,F)$ is called a {\it Finsler manifold} and $g_{ij}$ is called the fundamental metric tensor.}
	\end{definition}

\vspace{0.15in}

The notion of Finslerian geodesics extends to that of the
Riemannian geodesics.  The Finslerian geodesic equation
involves the spray coefficients which generalize the Riemannian Christoffel symbols. 

\vspace{0.15in}

\begin{definition}[{Finsler geodesic \cite[\boldmath p. $78$]{SSZ}}]
		\textnormal{A smooth curve  in a Finsler manifold $(M,F)$ is a {\it geodesic}, if it has constant speed and is locally length minimizing. Thus, a geodesic in a Finsler manifold $(M,F)$ is a curve
			$\gamma:I=(a,b)\rightarrow M$ with $F(\gamma(t), \dot{\gamma}(t))=$ constant and for any $t_0\in I$, there is a small number $\epsilon > 0$, such that
			$\gamma$ is length minimizing on $[t_0-\epsilon,t_0+\epsilon]\cap I$.}
	\end{definition}
 
\vspace{0.15in}

\begin{definition}[{Geodesic equation, Spray coefficients  \cite[\boldmath p. $78, 79$]{SSZ}}]
 
	A smooth \\curve $\gamma$ in a Finsler manifold $(M,F)$ is a geodesic, if and only if in local coordinates $\gamma(t)=(x^i(t))$ satisfies the following second order non linear differential equation:
	\begin{align}\label{2.2a}
		\frac{d^2x^i(t)}{dt^2}+{G}^i\left(x^i,\frac{d x^i}{dt}\right)=0, \qquad 1 \le i \le n,
	\end{align}
	where $G^i=G^i(x,y)$ are local functions on $TM$ defined by
	\begin{align}\label{eq1.2}
		{G}^i=\frac{1}{4}{g}^{i{\ell}}\left\{\left[{F}^{2}\right]_{x^ky^{\ell}}y^k-	\left[{F}^{2}\right]_{x^{\ell}}\right\}.
	\end{align}
	The coefficients $G^i$ are called the {\it spray coefficients} and the quantity $G:=y^i\frac{\partial}{\partial x^i}+ G^i\frac{\partial}{\partial y^i}$ is called {\it spray} on $M$.
\end{definition}

\vspace{0.15in}	

Now we discuss the notion of Riemann curvature tensor in the Finslerian settings in order to define the flag and the Ricci curvature.

\vspace{0.15in}	

\begin{definition}[\textbf{The Riemann curvature tensor \cite[\boldmath p. $50$]{CXSZ}}] \label{def2}  
 	Let $(M^n,F)$ be a Finsler manifold. The Riemann curvature tensor on $M$ is a mapping $R={{R}_\xi:T_xM \rightarrow T_xM}$ given by
 ${R}_{\xi}(u)={R}^{i}_{k}(x,\xi)u^{k} \frac{\partial}{\partial x^i}$,
 	\ $u=u^k\frac{\partial}{\partial x^k}$,\ where ${R}^{i}_{k}={R}^{i}_{k}(x,\xi)$ denote the coefficients of the Riemann curvature tensor of the Finsler metric $F$ and are given by
 	\begin{equation}\label{eqn2.2.10}
 			{R}^{i}_{k}=2\frac{\partial{G}^i}{\partial	x^k}-\xi^j\frac{\partial^2{G}^i}{\partial x^j\partial	\xi^k}+2{G}^j\frac{\partial^2 {G}^i}{\partial \xi^j\partial	\xi^k}-\frac{\partial{G}^i}{\partial \xi^j}\frac{\partial {G}^j}{\partial \xi^k},
 	\end{equation}
  for $1\leq i,j,k\leq n.$
 	\end{definition}

\vspace{0.15in}

 \noindent	The flag curvature $K=K(x,\xi,P)$, extends the concept of sectional curvature from Riemannian to Finsler geometry, independent of whether employing the Berwald, Chern, or Cartan connection.

\vspace{0.15in}

\begin{definition}[\textbf{Flag Curvature \cite[\boldmath p. $98$]{SSZ}}]

    Let $P \subset T_{x} M$ be a tangent plane. For a vector $y \in P \backslash\{0\}$, define

$$
\mbox{K}(P, y):=\frac{\mbox{g}_{y}\left(\mbox{R}_{y}(u), u\right)}{\mbox{g}_{y}(y, y) \mbox{g}_{y}(u, u)-\mbox{g}_{y}(y, u) \mbox{g}_{y}(y, u)},
$$

where $u \in P$ such that $P=\operatorname{span}\{y, u\}$. We can easily show that $\mbox{K}(P, y)$ is independent of $u \in P$ with $P=\operatorname{span}\{y, u\}$. The number $\mbox{K}(P, y)$ is called the flag curvature of the flag $(P, y)$ in $T_{x} M$.
\end{definition}







\vspace{0.15in}	

\begin{definition}[\textbf{The Ricci curvature \cite[\boldmath p. $98$]{SSZ}}]
Define

$$
{\operatorname { R i c }}(y):=\sum_{i=1}^{n} R_{i}^{i}(y).
$$

  Ric is a scalar function on $TM^{0}$, which satisfies the following homogeneity property:

$$
{\operatorname { R i c }}(\lambda y)=\lambda^{2} 
{\operatorname { R i c }}(y), \quad \lambda>0 .
$$

We call Ric the Ricci curvature. 
\end{definition}

\vspace{0.15in}

The notion of Einstein  Finsler manifolds have been studied  extensively, see for eg., \cite{xz1}.

\vspace{0.15in}	

	\begin{definition} [\textbf{Einstein manifold \cite[\boldmath p. $118$]{SSZ}}] \label{Einstein}
		Let $(M^n,F)$ be a Finsler manifold of dimension $n$, then $M$ is said to be an {\it Einstein} manifold if  
                 $$\operatorname{Ric}(x,y)=\lambda(x) F^2 (x,y).$$
        By differentiating twice with respect to $y_i$ and $y_j$, we infer        
		\begin{equation}
			\operatorname{Ric}_{ij}=\lambda(x)g_{ij},
		\end{equation}
		where $\lambda$ is a scalar function on $M$.
$(M,F)$ is said to be of  Ricci constant, if the function  $\lambda$ is constant.  
	\end{definition}
 
\vspace{0.15in}	
 For  Riemannian manifolds of dimension $\geq 3$,   Schur's lemma ensures that every Riemannian Einstein metric is necessarily of constant
	Ricci curvature. However, it is {\it not known} whether such a Schur 
	lemma holds for Einstein Finsler metrics in general.  But, if we restrict 
	Finsler metrics  to those of Randers type, then Schur's lemma holds for dimension $\geq 3$. See \cite{boa}, p. $216$, for more details.\\




 
Finsler manifolds differ drastically
from Riemannian manifolds and 
 have rich geometry. Distortion, $S$-Curvature,
 Berwald curvature, $E$-curvature etc., measure how far a Finsler manifold is away from being Riemannian. 
 We need these quantities for our geometric analysis and now we are probing into them.
\vspace{0.15in}	

 {\it A Finsler measure space},
 $(M^n,F,d\mu)$ is a Finsler manifold equipped with volume 
 as measure.

\begin{definition}[\textbf{Distortion,  \cite[\boldmath p. $117$]{SSZ}}]
     Suppose $(M^n,F,d\mu)$ is a Finsler measure space. 
  Let $\{b_i\}_{i=1}^n$ be an arbitrary basis for $T_xM$ 
  and $\{\theta^i\}_{i=1}^n$ be its dual basis for $T_x^\ast M$. Then we can write the volume form as $d\mu ={\sigma}_{\mu}(x) \; \theta^1\wedge...\wedge\theta^n $. 
  For a vector $y \in T_xM_0$, the distortion of a Finsler metric $F$ with measure $\mu$ is defined by 
     $$\tau_{\mu}(x,y):=\ln \frac{\sqrt{\operatorname{det}(g_{ij}(y))}}{\sigma_{\mu}(x)},$$
where $g_{ij}(y):=g_y(b_i,b_j).$
 \end{definition}
 
 \vspace{0.15in}

\begin{definition}[\textbf{S-Curvature \cite[\boldmath p. $118$]{SSZ}}]
$(M^n,F,d\mu)$ be a Finsler measure space.
    For $x\in M$ and a vector $y\in T_x M_0$, let $\gamma=\gamma(t)$ be the geodesic with $\gamma(0)=x$ and $\dot{\gamma}(0)=y$. Then
    $$S_{\mu}(x,y):=\frac{d}{dt}\left[\tau(\gamma(t),\dot{\gamma}(t))\right]\Bigr|_{t=0},$$
is called the $S$-curvature of the Finsler metric $F$
with distortion ${\tau}_{\mu}$.
\end{definition}

 \vspace{0.15in}	
 
	\begin{definition}[\textbf{The Berwald curvature tensor \cite[\boldmath p. $127$]{FMSBC}}]\label{Ber}
		Let $(M^n,F)$ be a Finsler manifold.  For a non-zero tangent vector $y\in T_xM$, in local coordinate system, define $B_y:T_xM \times T_xM \times T_xM \rightarrow T_xM$  by
			\begin{align}
				B_y(u,v,w):=B^i_{jk{\ell}}(x,y)u^jv^kw^{\ell}\frac{\partial}{\partial x^i},
			\end{align}
   where       
   \begin{align}
				B^i_{jk{\ell}}=[G^i]_{y^jy^k y^{\ell}},\quad  u=u^i\frac{\partial}{\partial x^i},\quad v=v^i\frac{\partial}{\partial x^i}, \quad w=w^i\frac{\partial}{\partial x^i},
			\end{align}
                for $1\leq i,j,k,l\leq n.$  Then 
			$B=\left\lbrace B_y|\ y\in TM^0  \right\rbrace  
			$ is called the 
   {\it Berwald curvature tensor} and the corresponding 
   coefficients $B^i_{jk{\ell}}$ are called as
{\it Berwald curvature}. 
   
	\end{definition}

A Finsler metric is called {\it Berwald metric}, if 
   $B \equiv 0.$

\vspace{0.15in}	

\begin{definition}[\textbf{E-Curvature \cite[\boldmath p. $118$]{SSZ}}]
     Let $$E_{ij} :=\frac{1}{2}S_{y^iy^j}=\frac{1}{2}\left[\frac{\partial G^s}{\partial y^s}\right]_{y^iy^j}.$$
     The symmetric tensor $E:=E_{ij}(x,y)\mathrm{d}x^i\otimes \mathrm{d}x^j $ is called the $E$ curvature tensor. $E=\{E_y|y\in TM^0\}$ is called the $E$-curvature or the mean Berwald curvature.
\end{definition}

A local formula for $E$-curvature is given by \cite{SSZ},

$$E_y(u,v)=\frac{1}{2}S_{y^iy^j}(y)u^iv^j=\frac{1}{2}\frac{\partial^3G^\gamma}{\partial y^i \partial y^j \partial y^\gamma}(y)u^iv^j=\frac{1}{2}B^{\gamma}_{ij\gamma},$$
 where $E_y(u,v):=E_{ij}(x,y)u^iv^j, $ $u=u^i\partial_i$, $v=v^i\partial_i$. \\

Finsler manifolds with mean Berwald curvature zero are
also called as {\it weakly Berwald} Finsler manifolds.
It should be noted that clearly, if $S =0$, then $E = 0$.
Moreover, the following result is well-known.

\vspace{0.15in}	

\begin{proposition}[Proposition 5.1.2 \textbf{\cite[\boldmath p. $90$]{VCRFG}}]
    For any Berwald metric, the $S$-curvature vanishes, $S \equiv 0$.
\end{proposition}

\vspace{0.15in}	

  There are numerous volume forms in Finsler geometry
 unlike in the Riemannian case. We will be working with the well-known volume forms, the Busemann Hausdorff, Holmes-Thompson, maximum and minimum. Now we define them.\\
 
Let $\mathbb{B}^n(1)$ denote the Euclidean unit ball in $\mathbb{R}^n$ and  ${\operatorname{Vol}(\mathbb{B}^n(1))}$  denote the canonical volume of the Euclidean ball.


\begin{definition}[{Busemann-Hausdorff Volume form \cite[\boldmath p. $22$]{SSZ}}]
 The \textit{Busemann-Hausdorff} volume form on a Finsler manifold $(M^n,F)$ is defined as:
  $\operatorname{dV}_{BH}=\sigma_{BH}(x)$ $\operatorname{dx}$, where
 \begin{equation}
 \sigma_{BH}(x)=\frac{\emph{Vol} (\mathbb{B}^n(1))}{\textnormal{Vol} \left\lbrace (\xi^i)\in T_xM: F(x,\xi)<1\right\rbrace}.
 \end{equation}
 \end{definition}
 
\vspace{0.15in}	
	
\begin{definition}
[\textbf{Holmes-Thompson Volume form \cite[\boldmath p. $6$]{sc}}] 
For a Finsler manifold $(M^n,F)$, the {\it Holmes-Thompson} volume form is defined as $\operatorname{dV}_{HT}=\sigma_{HT}(x)$ $\operatorname{dx}$, where 
            \begin{equation}\label{2.02}
				\sigma_{HT}(x)=\frac{1}{\operatorname{Vol}({\mathbb{B}}^n(1))}\int_{F(x,y)<1}\det(g_{ij}(x,y))\operatorname{dy}.
			\end{equation}
\end{definition}

   \vspace{0.15in}	
   
   \begin{definition}[\textbf{Max and Min Volume forms \cite[\boldmath p. $727$]{BYW}}]
		\textnormal{The {\it maximum} volume form $\operatorname{dV}_{max}$ and the {\it minimum} volume form $\operatorname{dV}_{min}$ of a Finsler metric $F$ is defined as,
			\begin{equation}\label{maxmin}
				\operatorname{dV}_{\max}=\sigma_{\max}(x) \operatorname{dx}, \quad  \operatorname{dV}_{\min}=\sigma_{\min}(x) \operatorname{dx},
			\end{equation}
			where $\sigma_{\max}(x)=\max\limits_{y \in I_x}\sqrt{\det(g_{ij}(x,y))}$ \quad and \quad $\sigma_{\min}(x)=\min\limits_{y\in I_x}\sqrt{\det(g_{ij}(x,y))}$;  \\$I_x=\left\lbrace y\in T_xM| F(x,y)=1\right\rbrace$ is the indicatrix of the Finsler metric $F$ at the point $x$.}
\end{definition}

\vspace{0.15in}

	\section{\textbf{Admissible class of Einstein Finsler warped product manifolds of finite volume}}\label{sec3}
In this section, we explore the concept of the warped product of two Finslerian manifolds, which is a generalization of the analogous Riemannian concept of warped product manifold. Then we 
recall the relationships between  Ricci curvatures and scalar curvatures of the warped product, its base and fiber manifold. Our main theorems rely heavily on these computations. See subsection
\ref{sec3.1}. Further, in subsection \ref{sec3.2}, we establish that 
Finsler warped products of finite Holmes-Thompson/maximum/minimum volume are of finite volume in the same category. Finally, in  subsection 
\ref{sec3.3}, we confirm the existence of Einstein Finsler warped
product manifolds of finite volumes, in the aforementioned categories, thereby establishing the admissible class
of  Einstein  Finsler warped product manifolds of finite volume.

  \subsection{ Einstein Finslerian warped product manifolds}\label{sec3.1}

   We begin with the definition of warped product Finsler manifolds.
	
\begin{definition} [\textbf{Warped product Finsler manifolds \cite[\boldmath p. $10$]{wlt}}]
  On the product manifold $M=M_1\times M_2$, we consider a function on $TM_1^0\times TM_2^0$, where $TM_i^0=TM_i\backslash \{0\}$, $i=1,2$, defined as
	\begin{equation}\label{FM}
		F^2(x_1,y_1,x_2,y_2)=F_1^2(x_1,y_1)+f^2(x_1)F_2^2(x_2,y_2),
	\end{equation}
	for $(x_1,y_1,x_2,y_2)\in TM_1^0\times TM_2^0$. The function $F$ is smooth and positively homogeneous. For the function $F$, we have
	\begin{align}
		&\nonumber\mathbf{g}_{ab}(x_1,y_1,x_2,y_2)=\Bigg(\frac{1}{2}\frac{\partial^2F^2(x_1,y_1,x_2,y_2)}{\partial \mathbf{y}^a\partial\mathbf{y}^b} \Bigg)\\&\quad\quad\quad\quad\quad\quad\quad\quad=\Bigg(\begin{matrix}
			g_{ij}(x_1,y_1) & 0\\
			0 & f^2(x_1)g_{\alpha\beta}(x_2,y_2)
		\end{matrix}\Bigg),
	\end{align}
	where 
	\begin{equation*}
		g_{ij}(x_1,y_1)=\frac{1}{2}\frac{\partial^2F_1^2(x_1,y_1)}{\partial y_1^i\partial y_2^j},\quad \quad  g_{\alpha\beta}(x_2,y_2)=\frac{1}{2}\frac{\partial^2F_2^2(x_2,y_2)}{\partial y_1^\alpha\partial y_2^\beta},
	\end{equation*}
	and $\mathbf{y}^a=(y_1^i,y_2^\alpha),\:\mathbf{y}^b=(y_1^j,y_2^\beta),\:\mathbf{g}_{ij}=g_{ij},\:\mathbf{g}_{\alpha\beta}=g_{\alpha\beta}$. The set of indices satisfy:
	\begin{align*}
		&i,j,k,...\in\{1,2,...,n_1\},\quad \alpha,\beta,\gamma,...\in\{1,2,...,n_2\},\quad a,b,c,...\in\{1,2,...,n_1+n_2\}.
	\end{align*}
	Thus $(M, F)$ defines a Finsler manifold and is called the {\it Finsler warped product} of  $M_1$ and $M_2$, and is denoted by $M=M_1\times_f M_2.$  $M_1$ is called  the base manifold and $M_2$ is called as the fiber manifold.  

 \end{definition}
	
\bigskip	
	
{\bf Note:} The warped product is called as {\it trivial}, if the warping function is constant. In this case, the product will
be termed as {\it Riemannian product} of two Finsler manifolds.\\

The following result by Tokura et al., \cite{wlt},
 which formulates the Ricci curvature of the warped product of Finslerian manifolds lays the foundation for our 
 research in this paper.
 In what follows, we will use   $\Delta f= - \mbox{tr}_g H^f=- \mbox{div}(\nabla f)$, where $H^f$ is the Hessian of the function $f$.
 
 \vspace{0.15in}
 
	\begin{proposition}[Proposition 3.6 \textbf{\cite[\boldmath p. $12$]{wlt}}]\label{f101}
		Let $M=M_{1}^{n_1}\times_{f}M_{2}^{n_2}$ be a Finslerian warped product manifold. If $(M_{1}^{n_1},g_1)$ is a Riemannian manifold, then
		\begin{align}
			 \operatorname{Ric}_{ij}&=(\operatorname{Ric}^{M_1})_{ij}-\frac{{n_2}}{f}\bigg(\frac{\partial^{2}f}{\partial x_1^i\partial x_2^j}-G_{ij}^k \frac{\partial f}{\partial x_1^k}\bigg),\label{f2}\\
			\operatorname{Ric}_{i\alpha}&=0,\label{f3}\\
				\operatorname{Ric}_{\alpha \beta}&= \operatorname{Ric}_{\alpha\beta}^{M_2}+\frac{1}{f} y_1^h\frac{\partial f}{\partial x_1^h}G_{\alpha\beta\gamma}^{\gamma}-[- f\Delta f + (n_2 - 1) g_1(\nabla f,\nabla f)]g_{\alpha\beta},\label{f4}
		\end{align}
		where $\nabla f$ and $\Delta f$ represent the gradient and the Laplacian of $f$ on the base manifold respectively, and $(\operatorname{Ric}^{M_1})_{ij}$ is the component of the Riemannian  Ricci curvature on $M_1$.
	\end{proposition}

\bigskip

 Hence, Einstein Finslerian warped products satisfy:

	\begin{corollary}\label{f5} Let $M = M_{1}^{n_1}\times_{f}M_{2}^{n_2}$ be a Finslerian warped product manifold, where $(M_{1}^{n_1},g_1)$ is a Riemannian manifold , then $M$ is Einstein Finslerian warped product with $\operatorname{Ric}_{ab}=\lambda g_{ab}$ if and only if  
		\begin{align}
			(\operatorname{Ric}^{M_1})_{ij}&=\lambda g_1{_{ij}} + \frac{{n_2}}{f}\bigg(\frac{\partial^{2}f}{\partial x_1^i\partial x_2^j}-G_{ij}^k \frac{\partial f}{\partial x_1^k}\bigg),\label{f6}\\
			\operatorname{Ric}_{\alpha\beta}^{M_2}=&[\lambda f^2-f\Delta f + (n_2-1) g_1(\nabla f,\nabla f)]g_{\alpha\beta}\label{f7}\\
			&-\frac{1}{f}y_1^h\frac{\partial f}{\partial x_1^h}G_{\alpha\beta\gamma}^{\gamma},\nonumber
		\end{align}
  where, $\lambda$ and $\mu$ respectively, are Ricci scalar functions on $M$ and fiber of $M$, viz., $(M_2^{n_2}, F_2)$ respectively. 
	\end{corollary}
	
	\begin{proof}
		Suppose $M$ is Einstein, then
		\begin{equation}\label{f122}
			\operatorname{Ric}_{ab}=\lambda g_{ab}.   
		\end{equation}
		Now using equations (\ref{f2}) and (\ref{f4}), we obtain  (\ref{f6}) and (\ref{f7}). 
	\end{proof}

\vspace{0.15in}

Consequently, we obtain  a corollary that will significantly aid  our forthcoming calculations.

	\begin{corollary}\label{f20}
		If $M_{1}^{n_1}\times_{f}M_{2}^{n_2}$ is a Finslerian warped product manifold,
  where $(M_{1}^{n_1},g_1)$ is a Riemannian manifold and $(M_2^{n_2},F_2)$ is a weakly Berwald Finsler manifold, then $M$ is Einstein Finsler warped product with $\operatorname{Ric}_{ab}=\lambda g_{ab}$ if and only if

		\begin{align}
		(\operatorname{Ric}^{M_1})_{ij}=\lambda g_1{_{ij}}+\frac{{n_2}}{f}H_{ij}^f,\label{f8}&&\\
			 (M_2,g_{M_2})\,\text{is  Einstein with }\, \operatorname{Ric}_{\alpha\beta}^{M_2}=\mu g_{\alpha\beta},\label{f9}&&\\
		 \mu=\lambda f^2 - f \Delta f + (n_2-1)  g_1(\nabla f,\nabla f).\label{f10}
		\end{align}
  
\end{corollary}
	
	\begin{proof}
		As $(M_1, g_1)$ is a Riemannian manifold, the spray coefficients
		$G_{ij}^k$ actually reduce to Christoffel symbols  $\Gamma^{k}_{ij}$. And the expression 
		$\bigg(\frac{\partial^{2}f}{\partial x_1^i\partial x_2^j}-G_{ij}^k \frac{\partial f}{\partial x_1^k}\bigg)$ is a coordianate expression for 
		the Hessian of $f$. Hence, we obtain (\ref{f8}). As fiber manifold is weakly Berwald, $G_{\alpha\beta\gamma}^{\gamma}=0$. Thus we get (\ref{f10}).
	\end{proof}
	
	\vspace{0.15in}

 It should be noted that, we recover Ricci curvatures of Riemannian warped   product manifolds, in particular, Corollary 3 \cite{kim}, about Einstein warped manifolds.

\bigskip

Next, we will obtain relations between $\operatorname{Scal}_M$, $\operatorname{Scal}_{M_1}$ and $\operatorname{Scal}_{M_2}$. These relationships will be useful in finding obstructions to 
the existence of Einstein  warped product manifolds.

\vspace{0.15in}

\begin{corollary}
If $M_{1}^{n_1}\times_{f}M_{2}^{n_2}$ is an Einstein Finslerian warped product manifold with 
$\operatorname{Ric}_{ab}=\lambda g_{ab}$, 
where $(M_{1}^{n_1},g_1)$ is a Riemannian manifold and $(M_2^{n_2},F_2)$ is a weakly Berwald Finsler manifold.
Then the scalar curvatures 
 of warped product, base and fiber satisfy:
\begin{equation}\label{1.4}
			\operatorname{Scal}_{M_1}=n_1\lambda- n_2 \frac{\Delta f}{f},
		\end{equation}
 \begin{equation}\label{1.5}
			\operatorname{Scal}_{M_2}=n_2\mu,
		\end{equation}
  \begin{equation}\label{1.6}
\operatorname{Scal}_{M_1}+\dfrac{\operatorname{Scal}_{M_2}}{f^2}=n_2(n_2 -1)\dfrac{\|\nabla f\|^2}{f^2}-2n_2\frac{\Delta f}{f}+\operatorname{Scal}_M .		
	\end{equation}
    
\end{corollary}
\begin{proof}
Now contracting the first two equations of the aforementioned corollary, we confirm
		
			$$\operatorname{Scal}_{M_1}=n_1\lambda- n_2 \frac{\Delta f}{f},$$

		and	$$\operatorname{Scal}_{M_2}=n_2\mu.$$
		
	If we use  \eqref{f10},  \eqref{1.4} and \eqref{1.5}, we conclude
	\begin{equation*}
\operatorname{Scal}_{M_1}+\dfrac{\operatorname{Scal}_{M_2}}{f^2}=n_2(n_2 -1)\dfrac{\|\nabla f\|^2}{f^2}-2n_2\frac{\Delta f}{f}+\operatorname{Scal}_M .		
	\end{equation*}
	\end{proof}

We will need the following important observation about Einstein
Finslerian warped product manifolds which tells us that the 
dimension of non-trivial Finslerian manifolds must be at least $4$.

\vspace{0.15in}

 \begin{proposition}\label{prop5.1}
     Let $(M_1^{n_1},g_1)$ be a compact Riemannian manifold,  the base, $f$ as the  warping function, and $(M_2^{n_2},F_2)$ be a weakly Berwald Finsler manifold,  the fiber. Then, for an Einstein warped product manifold 
     $M = M_1\times_f M_2 $ with $n_1=1$ or $n_2=1$  and a non-negative (or non-positive) Ricci scalar $\lambda(x_1,x_2)$, the warped product is simply a Riemannian product of two Finslerian manifolds, and is  in fact, also a weakly Berwald manifold.
 \end{proposition}

 \begin{proof}
     We prove it in two cases:

 {\bf Case 1 : $n_1=1.$}\\

As $n_1=1$ so $(\operatorname{Scal}_{M_1})_{ij}=0.$ If we take trace in \eqref{f8}, then we obtain
$$\frac{\lambda(x_1,x_2)}{n_2}=  \frac{\Delta f(x_1)}{f (x_1)}.$$

Observe that from the last equation, if $\lambda(x_1,x_2)$ is non-negative (or non-positive), then $f(x_1)$ is constant,
as the base manifold is compact.

{\bf Case 2 : $n_2=1.$}\\

As $n_2=1$ so $(\operatorname{Scal}_{M_2})_{ij}=0.$ If we take trace in \eqref{f9} then we obtain,
   $$\mu(x_2)=0.$$

   Thus, by \eqref{f10} we get
   $$\lambda(x_1,x_2) f(x_1) = \Delta f(x_1).$$

Hence, by the similar argument as in case $1$, $f(x_1)$ is constant.\\ 
In both the cases, $M$ is the Riemannian product of two Finsler manifolds and is weakly Berwald by Corollary \ref{weakly}
proved in \S \ref{sec4}
\end{proof}

\vspace{0.15in}

Thus, we conclude that, the dimension of Finslerian warped product can not be $1, 2,$ or $3$.

\vspace{0.15in}

 \begin{corollary}
Let $M = M_1^{n_1}\times_f M_2^{n_2} $ be as in the above proposition.
Then, for  Einstein  Finslerian manifold $M$, the warped product will be nontrivial, if $\operatorname{dim}M \geq 4$ and the Ricci scalar $ \lambda(x_1,x_2)$ is non-negative (or non-positive).

 \end{corollary}

\vspace{0.15in}

 \begin{remark}\label{dim}
     In view of the above corollary, 
in the sequel we will work with $n_1, n_2 \geq  2$ 
and $n \geq 4$.
 \end{remark}

\subsection{Finsler warped product manifolds of finite volume}\label{sec3.2}

\noindent

{\bf Warped product of Finsler manifolds of finite Holmes-Thompson, maximum and minimum Volume: }

 \vspace{0.15in}
 
	\begin{theorem}\label{ht}
		 Let $(M_1^{n_1},F_1)$  and  $(M_2^{n_2},F_2)$ be base and fiber  Finsler manifolds respectively, of the warped product $M=M_1\times_f M_2$. Suppose that $M_1, M_2$ have  
finite Holmes-Thompson volume, then the warped product $M=M_1\times_f M_2$ also has finite Holmes-Thompson volume provided that the warping function $f$ 
  is bounded, with $a \leq f \leq b$ for some $a, b > 0$
  and $b < \infty$.
  \end{theorem}

	\begin{proof}
	From Lemma $3.1$ of \cite{sc} it follows that,
		\begin{equation}\label{ss2}
			\operatorname{Vol}_{\operatorname{\operatorname{HT}}}(M^n,F)= \frac{1}{\operatorname{Vol}(\mathbb{B}^n(1))}\int_{BM^1}\det g_{ab}\;\operatorname{dx}\;\operatorname{dy},
		\end{equation}
            where $BM^1=\{(x,y)|F(x,y)\leq 1\}$. Then
            clearly,
            \begin{equation}\label{BHMAX}
                F_1^2(x_1, y_1) +f^2(x_1) F_2^2(x_2, y_2)\leq 1,   
                \end{equation}
		implies that 
        \begin{equation}\label{BHMAX1}
            BM^1\subseteq BM_1^1 \times BM_2^{1/a}.       
        \end{equation}
		
		Therefore, 
		\begin{eqnarray}
			\int_{BM^1}\det g_{ab} \; \operatorname{dx} \;\operatorname{dy} & \leq &b^{2n_2}\int_{BM_1^1 \times BM_2^{1/a}} \det (g_{ij})\det (g_{\alpha\beta})\; \operatorname{dx} \; \operatorname{dy}. \\ \nonumber
		\end{eqnarray}
Equivalently, by using Fubini-Tonelli Theorem,
\begin{align}
      \frac{1}{\operatorname{Vol}(\mathbb{B}^n(1))\operatorname{Vol}(\mathbb{B}^{n_1}(1))\operatorname{Vol}(\mathbb{B}^{n_2}(1/a))}&\int_{BM^1}\det g_{ab} \; \operatorname{dx} \;\operatorname{dy}\leq\\ \nonumber
      &\frac{b^{2n_2}}{\operatorname{Vol}(\mathbb{B}^n(1))}\operatorname{Vol}_{\operatorname{HT}}(M_1^{n_1},F_1)\operatorname{Vol}_{\operatorname{HT}}(M_2^{n_2},F_2).  
      \end{align}
         The last step follows as
         $$\operatorname{Vol}_{\operatorname{HT}}(M^{n_2},F_2)=\frac{1}{\operatorname{Vol}(\mathbb{B}^{n_2}(1/a))}\int_{BM_2^{1/a}} \det g_{ab} \; \operatorname{dx} \; \operatorname{dy}.$$
         Consequently,
         $$\operatorname{Vol}_{\operatorname{HT}}(M^n,F)\leq\frac{b^{2n_2}\operatorname{Vol}(\mathbb{B}^{n_1}(1))\operatorname{Vol}(\mathbb{B}^{n_2}(1/a))}{(\operatorname{Vol}_{\operatorname{HT}}(M_1^{n_1},F_1))(\operatorname{Vol}_{\operatorname{HT}}(M_2^{n_2},F_2))}.$$
         Thus, the assertion follows.
	\end{proof}

	

 \vspace{0.15in}
 
	\begin{theorem}\label{max}
		If $(M_1^{n_1},F_1)$  and $(M_2^{n_2},F_2)$ be base and fiber Finsler manifolds respectively, with finite maximum (minimum)  volume, then the warped product $M=M_1\times_f M_2$ has also  finite maximum (minimum) volume, provided that warping function $f$ is bounded as $a \leq f \leq b$ for some $a, b > 0$.	  
	\end{theorem}
	\begin{proof}
		Clearly,
		\begin{equation}
			\sigma_{\max}(x)= \max\limits_{y \in I_{x} M}\sqrt{\det(g_{ab}(x ,y))}
   = \max\limits_{y \in{BM^1}}\sqrt{\det(g_{ab}(x ,y))}.
		\end{equation}
	Thus we  confirm that,	
		\begin{dmath}
			\operatorname{Vol}_{\max}(M,F)= \int_{M} \max\limits_{(y_1,y_2) \in{BM^1}}\sqrt{\det(g_{ab}(x_1,x_2,y_1,y_2))}\; \operatorname{dx}.
        \end{dmath} 
            Then by \eqref{BHMAX} and \eqref{BHMAX1} we have,    
         
       \begin{dmath}  
\operatorname{Vol}_{\max}(M,F)
             \leq   b^{2n_2} \left( \int_{M_1} \max\limits_{y_1 \in BM_1^1}\sqrt{\det(g_{ij}(x_1,y_1))}\operatorname{dx}_1\right)\\
            \left ( \int_{M_2}\max\limits_{y_{2} \in BM_2^{1/a}}\sqrt{\det(g_{\alpha\beta}(x_2,y_2))} \operatorname{dx}_2 \right).
        \end{dmath}    
         


		For minimum volume form the result can be proved similarly.
	\end{proof}
	
	\vspace{0.2in}

\begin{corollary}If $(M_i^{n_i}, F_i),$  for $i = 1,2$ be base and fiber Finsler manifolds respectively, of finite Holmes-Thompson volumes, and if the warped product $(M,F)$ is a Randers manifold,  then  $M$ is also of finite  Busemann Haudorff volume.
  \end{corollary} 
  
\begin{proof}
 This follows as by  Remark $3.2$ (ii)  p.$5$
\cite{sc},
it follows that $\operatorname{Vol}_{BH}(M,F)\leq \operatorname{Vol}_{HT}(M,F)$.
\end{proof}

\subsection{Existence  of Einstein Finsler warped product of finite volume}\label{sec3.3}
  Inspired by the  proof of Proposition $5$ of \cite{kim},
  we affirm the following result which  establishes 
  the admissible class of Einstein Finsler manifolds.\\
  
 We recall the following identity proved in Lemma
$4$ \textbf{\cite[\boldmath p. $2574$]{kim}}]:

\begin{align} 
			{\operatorname{div}(H^f)}(X)=\operatorname{Ric}(\nabla f,X)-\Delta (df)(X),\label{f14}
		\end{align}
  where $\Delta=d\delta+\delta d$ denotes the Laplacian on base manifold acting on differential forms.

  \vspace{0.15in}
	
	\begin{theorem}\label{exist}
		Let $(M_{1}^{n_1}, g_1)$ be a Riemannian manifold of finite volume and  $f$ be a non-constant, positive  smooth  bounded function on $M_1$, that is $a \leq f \leq b$. If  $f$ satisfies (\ref{f8}) for a scalar function $\lambda$ and for natural numbers $n_2\in\mathbb{N}$, then $f$ satisfies (\ref{f10}) for a Ricci scalar function $\mu$. Hence, for a finite volume Holmes-Thompson, maximum or minimum volume, weakly Berwald Einstein Finsler manifold  $({M_2}^{n_2}, F_2)$ with $\operatorname{Ric}_{\alpha\beta}^{M_2}=\mu g_{\alpha\beta}$, for  scalar function $\mu$, we can construct a finite volume Finsler warped product manifold $M=M_{1}^{n_1}\times_{f}M_{2}^{n_2}$ with $\operatorname{Ric}_{ab}=\lambda g_{ab}$.   
	\end{theorem}
\begin{proof}
		 Contracting both sides of (\ref{f8}), we have
		\begin{equation*}
			\operatorname{Scal}_{M_1}=n_1\lambda -\frac{n_2}{f}\Delta f.
		\end{equation*}
		
		The above equation imply that
		\begin{eqnarray}
			d\operatorname{Scal}_{M_1}= n_{1} d\lambda + \frac{n_2}{f^2}\{\Delta f df-fd(\Delta f) \}.\label{f11}
		\end{eqnarray}
		From the second Bianchi's identity, we conclude that
		\begin{eqnarray}
			d\operatorname{Scal}_{M_1}=2\mbox{div}(\operatorname{Ric}^{M_1}).\label{f12}
		\end{eqnarray}
		Now, equations (\ref{f11}) and (\ref{f12}) provide that
		\begin{eqnarray}
			\mbox{div}(\operatorname{Ric}^{M_1})=  \frac{n_1}{2} d\lambda + \frac{n_2}{2f^2}\{\Delta f df-fd(\Delta f) \}.\label{f13}
		\end{eqnarray}
		From proof of Proposition 5 of  \cite{kim} we  determine,
		\begin{eqnarray}
			\mbox{div}\left(\frac{1}{f}H^f\right)=-\frac{1}{2f^2}d(g_1(\nabla f,\nabla f))+\frac{1}{f}\mbox{div}(H^f).\label{f15}
		\end{eqnarray}
		From equations (\ref{f14}) and (\ref{f8}), we calculate
		\begin{equation}
			\mbox{div}(H^f)=\lambda df+\frac{n_2}{f} H^f - \Delta df .\label{f16}
		\end{equation}
		After using (\ref{f16}) in (\ref{f15}),  we derive	
		\begin{align}
			\mbox{div}\left(\frac{1}{f}H^f\right)&=\frac{1}{2f^2}\big[(n_2-1)d(g_1(\nabla f,\nabla f))+2\lambda f df\label{f18}\\
			&-2fd(\Delta f)\big].\nonumber
		\end{align}	
		Now taking divergent of both side of (\ref{f8}), we have
		\begin{eqnarray}
			\mbox{div}(\operatorname{Ric}^{M_1})=\mbox{div}\left(\frac{n_2}{f}H^f\right) + \mbox{div}(\lambda g_1)  .\label{f19}
		\end{eqnarray}
		Thus from equations (\ref{f13}), (\ref{f18}) and (\ref{f19}), we infer that

\begin{equation*}
		n_2	d[(n_2-1)(g_1(\nabla f,\nabla f))-(f \Delta f)+\lambda f^2]+  2f^2\operatorname{div}(\lambda g_1)- ( 2 n_1 + {n_2}) (d\lambda) f^2 ) =0.
		\end{equation*}
 Integrating the above equation yields that,
		\begin{equation}\label{valueofmu}
			\mu=(n_2-1)(g_1(\nabla f,\nabla f))-(f\Delta f)+\lambda f^2,
		\end{equation}
		for some scalar function $\mu$. Thus, we proved the first part of this theorem. Next, for given a finite volume Holmes-Thompson, maximum or minimum volume Finslerian manifold $(M_2^{n_2}, F_2)$ with $\operatorname{Ric}_{\alpha\beta}^{M_2}=\mu g_{\alpha\beta}$,  we can construct a finite volume Einstein Finslerian warped product manifold $M=M_{1}^{n_1}\times_{f}M_{2}^{n_2}$ with $\operatorname{Ric}_{ab}=\lambda g_{ab}$, by sufficient part of Corollary \ref{f20}, Theorem \ref{ht} and Theorem \ref{max}, provided  that warping function $f$ is bounded.
\end{proof}

\vspace{0.15in}

\begin{remark}
    \emph{Let} $\mathcal{M}=\{M=M_1^{n_1}\times_f M_2^{n_2}|\;(M_i^{n_i},F_i),i=1,2,\emph{of finite volume}\\ \emph{ (with aforementioned measures), } M\emph{ is Einstein and\;} a\leq f\leq b \}.$\\
    
    \emph{In summary, we have shown that the collection 
$\mathcal{M}$ constitutes the admissible class of all Finsler manifolds, and henceforth, we will work with them.}
\end{remark}

\section{Warped product of weakly Berwald Finsler manifolds}\label{sec4}

It is interesting to know whether the warped product of two 
weakly Berwald or Berwald Finslerian manifold is again
weakly Berwald or Berwald, respectively. We explore this in this section. 

 \vspace{0.15in}
 
	\begin{theorem}\label{Berwald product}
		If the base $(M_1^{n_1},F_1)$ and fiber $(M_2^{n_2},F_2)$
  respectively, of the warped product are Berwald Finsler manifolds, then the warped product $(M,F)$ is a Berwald Finsler manifold if and only if the warping function $f$ is constant.
	\end{theorem}
 
	\begin{proof}
		The geodesic spray  of the warped product Finsler metric \eqref{FM} are given by (\cite{wlt}, p. 11) 
  \begin{equation}\label{coeff}
		\mathbf{G}^i=G^i-\frac{1}{4}g^{ih}\frac{\partial f^2}{\partial x_1^h}F_2^2,\quad \mathbf{G}^\alpha=G^\alpha-\frac{1}{2}\frac{1}{f^2}y_1^h\frac{\partial f^2}{\partial x_1^h}y_2^\alpha.
	\end{equation}
		Now we use, \begin{align*}
			\mathbf{G}^a_b :=\frac{\partial \mathbf{G}^a}{\partial \mathbf{y}^b},\quad G^i_j:=\frac{\partial G^i}{\partial y_1^j},\quad G^\alpha_\beta:=\frac{\partial G^\alpha}{\partial y_2^\beta}.  
		\end{align*}
		Then the spray coefficients of warped product are given by,
		\begin{align*}
			\mathbf{G}^i_j=G^i_j-\frac{1}{4}\frac{\partial g^{ih}}{\partial y_1^j}\frac{\partial f^2}{\partial x_1^h}F_2^2,\quad\mathbf{G}^i_\beta= -\frac{1}{4}g^{ih}\frac{\partial f^2}{\partial x_1^h}\frac{\partial F^2}{\partial y_2^\beta}.\\
			\mathbf{G}^\alpha_j=\frac{1}{2}\frac{1}{f^2}y_2^\alpha\frac{\partial f^2}{\partial x_1^j},\quad\mathbf{G}^\alpha_\beta=G^\alpha_\beta+\frac{1}{2}\frac{1}{f^2}y_1^h\frac{\partial f^2}{\partial x_1^h}\delta^\alpha_\beta.
		\end{align*}
		Since we have,
		\begin{align}
			\mathbf{G}^i_{jk}=\frac{\partial}{\partial y_1^k}\Bigg(\mathbf{G}^i_j\Bigg)=G^i_{jk}-\frac{1}{4}\frac{\partial^2 g^{ih}}{\partial y^j_1\, \partial y_1^k}, \frac{\partial f^2}{\partial x_1^h} F_2^2. 
		\end{align}
		\begin{align}
			\mathbf{G}^i_{j\alpha}=\frac{\partial}{\partial y_2^\alpha}\Bigg(\mathbf{G}^i_j\Bigg)=-\frac{1}{4}\frac{\partial g^{ih}}{\partial y_1^j}\,\frac{\partial f^2}{\partial x_1^h} \frac{\partial F_2^2}{\partial y_2^\alpha}.
		\end{align}
		\begin{align}
			\mathbf{G}^i_{\alpha\beta }=\frac{\partial}{\partial y_2^\beta}\Bigg(\mathbf{G}^i_{\alpha }\Bigg)= -\frac{1}{4} g^{ih}\,\frac{\partial f^2}{\partial x_1^h} \frac{\partial F_2^2}{\partial y_2^\alpha}.  
		\end{align}
		\begin{align}
			\mathbf{G}^\alpha_{jk} =\frac{\partial}{\partial y_1^k}\Bigg(\mathbf{G}^\alpha_j\Bigg)=0.   
		\end{align}
		\begin{align}
			\mathbf{G}^\alpha_{\beta\gamma}= \frac{\partial}{\partial y_2^\gamma}\Bigg(\mathbf{G}^\alpha_\beta\Bigg) =G^\alpha_{\beta\gamma}. 
		\end{align}

Now we explicitly compute the  Berwald curvatures in  warped product Finsler manifold. 

Peyghan and Tayebi have already computed the Berwald curvatures of a doubly warped product Finsler manifold in 
Theorem $2$, \cite{MLA}. We compute them here  
for the sake of completeness.

		\begin{align}\label{eq11}
			\mathbf{G}^i_{jkl}=\frac{\partial}{\partial y_1^l} \Bigg(\mathbf{G}^i_{jk}\Bigg)= G_{jkl}^i-\frac{1}{4}\frac{\partial^3 g^{ih}}{\partial y_1^j\partial y_1^k\partial y_1^l}\,\frac{\partial f^2}{\partial x_1^h} F_2^2,
		\end{align}
		where $l\in \{1,2,3,...,n_1\}.$
		\begin{align}\label{eq2}
			\mathbf{G}^\alpha_{\beta\gamma\delta}=G^\alpha_{\beta\gamma\delta},   
		\end{align}
		where $\delta\in \{1,2,3,...,n_2\}.$
		
		\begin{align}\label{eq3}
			\mathbf{G}^i_{jk\alpha}=\frac{\partial}{\partial y_2^\alpha} \Bigg(\mathbf{G}^i_{jk}\Bigg)=-\frac{1}{4}\frac{\partial^2 g^{ih}}{\partial y_1^j\partial y_1^k}\,\frac{\partial f^2}{\partial x_1^h} \frac{\partial F_2^2}{\partial y_2^\alpha}.
		\end{align}
		
		\begin{align}\label{eq4}
			\mathbf{G}^i_{j\alpha\beta}=\frac{\partial}{\partial y_2^\beta} \Bigg(\mathbf{G}^i_{j\alpha}\Bigg)=-\frac{1}{4}\frac{\partial g^{ih}}{\partial y_1^j}\,\frac{\partial f^2}{\partial x_1^h} \frac{\partial^2 F_2^2}{\partial y_2^\alpha\partial y_2^\beta}.  
		\end{align}
		\begin{align}\label{eq5}
			\mathbf{G}^i_{\alpha\beta \gamma}=\frac{\partial}{\partial y_2^\gamma}\Bigg(\mathbf{G}^i_{\alpha \beta}\Bigg) =-\frac{1}{4} g^{ih}\,\frac{\partial f^2}{\partial x_1^h} \frac{\partial^3 F_2^2}{\partial y_2^\alpha\partial y_2^\beta \partial y_2^\gamma}.  
		\end{align}
		
		
		\begin{align}\label{eq6}
			\mathbf{G}^\alpha_{ijk} =0
		\end{align}
		\begin{align}\label{eq7}
			\mathbf{G}^\alpha_{ij\gamma} =0 
		\end{align}
		\begin{align}\label{eq8}
			\mathbf{G}^\alpha_{i\beta\gamma}=0. 
		\end{align}
		Clearly,  $(M,F)$ is Berwald if and only if  all the Berwald  curvatures 
		vanish. Thus  \eqref{eq11}-\eqref{eq8}
		all vanish. Also observe that 
\eqref{eq4} vanishes if and only if the warping function $f$ is constant.  
 Hence, the result follows.
	\end{proof}
	
	\vspace{0.1in}

 \begin{corollary}\label{weakly}
    If the base $(M_1^{n_1},F_1)$ and fiber $(M_2^{n_2},F_2)$
    respectively, of the warped product are weakly Berwald Finsler manifolds, then the warped product $(M,F)$ is a weakly Berwald Finsler manifold, if the warping function $f$ is constant.
     \end{corollary}

\vspace{0.15in}

From Theorem \ref{Berwald product}, Corollary \ref{weakly}, 
Theorem \ref{ht} and Theorem \ref{max}, we infer:
\vspace{0.15in}

	\begin{corollary}\label{cor4.2}
		If the base $(M_1^{n_1},F_1)$ and fiber $(M_2^{n_2},F_2)$
  respectively, of the warped product are Berwald or weakly Berwald Finsler manifolds of finite Holmes-Thompson, maximum or minimum volume, then the warped product $(M,F)$ is also Berwald or weakly Berwald Finsler manifold with respect to the same volume measure, if the warping function $f$ is constant.
  \end{corollary}

\section{Einstein Finsler warped product manifold with Ricci a non-positive scalar function}\label{sec5}
In this section, we investigate the obstructions to the 
existence of $(M, F) \in \mathcal{M}$, when
Ricci of $(M,F)$ is a non-positive scalar function.
We will be using the following generalized maximum principle
	which guarantees the existence of a maximizing sequence.
	The more details can be found in \cite{PRS}. 

\vspace{0.15in}

\begin{definition}[\textbf{Omori-Yau maximum principle \cite[\boldmath p. $486$]{PRS}}]
    A Riemannian manifold $(M,g)$ is said to satisfy 
	the {\it Omori-Yau maximum principle}, if for any smooth function $u\in C^2(M)$ with $\sup u < \infty$,  there exist a sequence $\{x_n\}$ such that
	\begin{enumerate}
		\item[\emph{(i)}] $\displaystyle{\lim_{n \to \infty}}u(x_n)={\sup}_{M} u,$\\
		\item[\emph{(ii)}] $||\nabla u(x_n)|| \leq \frac{1}{n},$ \\
		\item[\emph{(iii)}] $\Delta u(x_n)\ge \frac{1}{n}$, where $\Delta u = -\operatorname{tr}(\Delta^2 u)$.
	\end{enumerate}
\end{definition}
	
\vspace{0.15in}	

	\begin{theorem}[\textbf{\cite[\boldmath p. $486$]{PRS}}] \label{rbb}
		The Omori-Yau maximum principle holds on every complete Riemannian 
		manifold with Ricci curvature bounded below.
	\end{theorem}
	
\vspace{0.15in}

Now we prove the main theorem of this section.

\vspace{0.15in}

\begin{theorem}\label{main}
		Let the base $(M_1^{n_1}, g_1)$ be a finite volume complete, connected  Riemannian manifold, and the fiber  $(M_2^{n_2}, F_2)$ be a weakly Berwald Finsler manifold, with the warped product $M = M_1 \times_f M_2$. Suppose the warping function $f$ satisfies the following conditions: \emph{(i)}  $f$ is bounded as 
  $a \leq f \leq b$ \emph{(ii)} $||\nabla f|| \in L^1(M_1)$ 
  and (iii) $\nabla^2 f\geq l $, for some constant $ 0<l<\infty$. If $M$ is an Einstein Finsler manifold with constant non-positive scalar curvature, then the warping function $f$ is constant, and $M$ is simply the Riemannian product of two Finslerian manifolds which is also a weakly Berwald manifold. 
\end{theorem}
\begin{proof}
Availing  (\ref{f10}), 
 \begin{equation*}
		\mu=\lambda f^2- f\Delta f + (n_2-1)\|\nabla f\|^2.
	\end{equation*}
	By hypothesis (iii) of the theorem, 
 (\ref{f8}), and as  $f$ is bounded, Ricci $M_1$ is bounded below. Also base manifold is complete, by 
	Omori-Yau maximum Principle,
	there exists $\{x_k\}$ and $\{y_k\}$ such that 
	\begin{equation}
		\displaystyle{\lim_{k\to \infty }} f(x_k)=\sup f, \;\displaystyle{\lim_{k\to \infty }} ||\nabla f(x_k)||=0, \;\displaystyle{\lim_{k\to \infty }}\Delta f(x_k)\ge 0
	\end{equation}
	and
	\begin{equation}
		\displaystyle{\lim_{k\to \infty }} f(y_k)=\inf f, \;\displaystyle{\lim_{k\to \infty }}||\nabla f(y_k)||=0, \;\displaystyle{\lim_{k\to \infty }}\Delta f(y_k)\le 0.
	\end{equation}
	As $\lambda \le 0$, we obtain
	\begin{equation}\label{eq1.1}
		\lambda(\sup f)^2-\mu \le \lambda (\inf f)^2-\mu.
	\end{equation}
	But, 
	$$\lim_{k \rightarrow \infty} f(x_k) \Delta f(x_k)=\sup f \lim_{k \rightarrow \infty} \Delta f(x_k)\ge 0.$$
	Therefore, by expression of $\mu$  by (\ref{f10}),	\begin{equation}\label{eq1.11}
		\lambda(\sup f)^2-\mu \ge 0.
	\end{equation}
	Similarly, 
	\begin{equation}
		\displaystyle{\lim_{k\to \infty}}f(y_k)\Delta f(y_k)\le 0.
	\end{equation}
	Consequently,
	\begin{equation}\label{eq1.21}
		\lambda(\inf f)^2-\mu \le 0.
	\end{equation} 
	Therefore,
	\begin{equation}\label{eq1.02}
		\lambda(\sup f)^2-\mu \ge \lambda(\inf f)^2-\mu.
	\end{equation}
	Hence, from \eqref{eq1.1} and \eqref{eq1.02} we get,
	\begin{equation}\label{kam}
		\lambda(\sup f)^2-\mu = \lambda(\inf f)^2-\mu.
	\end{equation}
	
	Now we consider the following cases:\\
	(i) $\lambda <0$ :\\ 
	By (\ref{kam}), $\sup f = \inf f$, which implies that warping function $f$ is constant.\\
	
	(ii) $\lambda=0$ :
	
	By \eqref{eq1.11} and \eqref{eq1.21}, $\mu=0$.
	Therefore,
	\begin{equation}
		\lambda f^2- f\Delta f + (n_2-1)\|\Delta f\|^2= \mu,
	\end{equation}
	implies that	
	\begin{equation}\label{lap}
		f\Delta f = (n_2-1)\|\nabla f\|^2 \ge 0.
	\end{equation}
Hence, by hypothesis (iii),  $f \Delta f$  
 is integrable.
	Note that by (\ref{lap}), 
 $$ \mbox{div}(f \nabla f) = ||\nabla f||^2 - f \Delta f =- (n_2 - 2) ||\nabla f||^2.$$ 
	 Thus by Gaffney's Stokes' Theorem, we obtain:
	$$\int_{M} \mbox{div}(f \nabla f) =-\int_{M} (n_2 - 2) ||\nabla f||^2 = 0.  $$
	This in turn implies that warping function $f$ must be constant, if $n_2 \geq 3$.\\\\
	Now suppose that $n_2 = 2$.  Then by (\ref{lap}),
 $f \Delta f = (n_2 - 1)  ||\nabla f||^2  = ||\nabla f||^2.$  
 Then integrating this and using $\Delta f^2$ yields that $\nabla f = 0$, that is warping function $f$ is constant on $M$.
Thus the final theorem is proved in view of Corollary \ref{weakly}.

 \end{proof}

We obtain an important corollary of this result for Randers manifold.

\vspace{0.15in}
	
	\begin{corollary}
	Let $(M^n,F), n \geq 3$, be an Einstein Berwald Randers  Finsler manifold of Ricci, a non-positive scalar function, and which
 satisfy the hypothesis (i), (ii), (iii) of the above corollary. 
 Then, $M$ is simply Riemannian product of two Finslerian manifolds and is a weakly Berwald manifold. 
	\end{corollary}
	\begin{proof}
		The result follows from the above theorem, as in this case Schur's lemma for Ricci curvature 
		implies that Randers manifold is actually of constant Ricci curvature.
		(See explanation after Definition \ref{Einstein}).
	\end{proof}

\bigskip

In the next theorem we consider compact Riemannian manifold 
as a base of the warped product and $\lambda$ as non-positive scalar function. In view of Remark \ref{dim},  it suffices to consider the  case when  $n_1, n_2 \geq  2$ and $n \geq 4$. 

\vspace{0.15in}
 
\begin{theorem}\label{mainf}
	Let the base $(M_1^{n_1},g_1)$ be a compact, connected Riemannian manifold, $f$ being a  warping function on the base manifold and $(M_2^{n_2},F_2)$ being a weakly Berwald Finsler manifold. Then for an Einstein warped product $M=M_1\times_f M_2$ with non-positive Ricci scalar function $\lambda$, the
 warping function $f$ is constant, and warped product  
 is simply a Riemannian product of two Finslerian manifolds;
 in fact is also a weakly Berwald manifold.
	\end{theorem}

\begin{proof}
    As the  base manifold is compact, warping function $f$ attains maximum and minimum. Hence,
	there exists $x_0$ and $y_0$ such that 
	\begin{equation}
		\displaystyle f(x_0)=\max f, \;\displaystyle\nabla f(x_0)=0, \;\displaystyle\Delta f(x_0)\ge 0,
	\end{equation}
	and
	\begin{equation}
		\displaystyle f(y_0)=\min f, \;\displaystyle\nabla f(y_0)=0, \;\displaystyle\Delta f(y_0)\le 0.
	\end{equation}
Now,
   \begin{equation}
   	\mu(y)= \{\lambda(x_0,y)f^2(x_0)-f(x_0)\Delta f(x_0)+(n_2 -1)||\nabla f||^2(x_0)\}.
   \end{equation}
	As $\lambda(x_0,y) \le 0$, we get
	\begin{equation}\label{eq1.141}
		\lambda(\max f)^2-\mu \le \lambda (\min f)^2-\mu.
	\end{equation}
	But, 
	$$ f(x_0) \Delta f(x_0)=\max f \; \Delta f(x_0)\ge 0.$$
	This yields by expression of $\mu$  (\ref{f10}),
	\begin{equation}\label{eq1.111}
		\lambda(\max f)^2-\mu \ge 0.
	\end{equation}
	Similarly, 
 \begin{equation}
	\mu(y)= \lambda(y_0,y)f^2(y_0)-f(y_0)\Delta f(y_0)+(n_2 -1)||\nabla f||^2(y_0).
\end{equation}
But, 
$$ f(y_0) \Delta f(y_0)=\max f  \Delta f(y_0)\le 0.$$
	Consequently,
	\begin{equation}\label{eq1.211}
		\lambda(\min f)^2-\mu \le 0.
	\end{equation} 
	Hence,
	\begin{equation}\label{eq1.021}
		\lambda(\max f)^2-\mu \ge \lambda(\min f)^2-\mu.
	\end{equation}
	Employing \eqref{eq1.141} and \eqref{eq1.021} we obtain,
	\begin{equation}\label{lam}
		\lambda(\max f)^2-\mu = \lambda(\min f)^2-\mu.
	\end{equation}
	
	Now we consider the following cases:
	
	(i) $\lambda <0$ :
	
	By (\ref{lam}), $\max f = \min f$ which implies that $f$ constant.
	
	(ii) $\lambda=0$ :
	
	By \eqref{eq1.111} and \eqref{eq1.211}, $\mu=0$.
	Thus,
	\begin{equation}
		\lambda f^2- f\Delta f + (n_2-1)\|\Delta f\|^2= \mu,
	\end{equation}
	implies that	
	\begin{equation}\label{kap}
		f\Delta f = (n_2-1)\|\nabla f\|^2 \ge 0.
	\end{equation}

	Hence, warping function $f$ is constant.
Thus, the conclusion follows in view of Corollary \ref{weakly}.
\end{proof}

\begin{corollary}
Let $(M_1^{n_1},g_1)$ be a compact Riemannian manifold
 and  $f$ be a  warping function on the base manifold and $(M_2^{n_2},F_2)$ be a weakly Berwald Finsler Manifold. Then 
 for the warped product $M=M_1\times_f M_2$ to be a legitimate Einstein Finsler warped product 
 manifold, Ricci must be a positive scalar function 
\end{corollary}

\vspace{0.15in}
 

 \begin{remark}
     Note that  proof of Theorem \ref{mainf} is entirely different from Kim and Kim's original proof. In particular, we provide an entirely different  proof of Kim and Kim's result,
     Theorem \ref{kk}.
 \end{remark}

\section{Einstein Finsler warped product manifold with Ricci a positive scalar function}\label{sec6}
In this section, we investigate the triviality of Einstein warped product manifold with Ricci scalar non-positive, and base a compact manifold. We also explore obstructions to  trivaliy in terms of the scalar curvatures of base, fiber and warped product. In  view of Proposition \ref{prop5.1}, the subsequent section we will focus on scenarios where $n_1,n_2$ are both are at least $2$, while $n$ exceeds $4$.

\vspace{0.15in}

\begin{theorem}
	Let $M=M_1^{n_1}\times_{f} M_2^{n_2}$ be an Einstein warped product of non-negative Ricci scalar  with the base a Riemannian manifold and the fiber a weakly Berwald Finsler manifold of  non-positive  scalar function at one point. Then the warped product $M=M_1\times_f M_2$ is weakly Berwald  Riemannian product of two 
 Finslerian manifolds, if any of the following holds:
	\begin{itemize}
		\item[\emph{(i)}] Warping function $f$ attains its maximum in base manifold.
		\item[\emph{(ii)}] Base manifold is compact.
	\end{itemize}
\end{theorem}

\begin{proof}
	Suppose $\Tilde{x_2} \in M_2$ is a point where the scalar curvature is non-positive, that is $\mu(\Tilde{x_2})<0$. Then by \eqref{f10}, we have  \newline
	\begin{equation*}
		-f(x_1)\Delta f(x_1) +(n_2-1)\|\nabla f\|^2(x_1)+\lambda (x_1,\Tilde{x_2}) f^2(x_1)=\mu(\Tilde{x_2}).
	\end{equation*}
	This yiels that,
	\begin{equation*}
		-f(x_1)\Delta f(x_1) +\lambda (x_1,\Tilde{x_2}) f^2(x_1)=\mu(\Tilde{x_2})-(n_2-1)\|\nabla f\|^2(x_1)\leq 0.
	\end{equation*}
	Hence, we affirm that
	\begin{equation*}
		f(x_1)\Delta f(x_1) \geq \lambda(x_1,\Tilde{x_2})f^2(x_1). 
	\end{equation*}		
As $f(x)$ is strictly positive function on base manifold and $\lambda$ is positive, we confirm
	\begin{equation*}
		\Delta f(x_1) \geq \lambda f(x_1) > 0.
	\end{equation*}
	This implies that the warping function 
$f$ is a subharmonic function on the base manifold. Using the maximum principle, $f$ is constant, if it attains its maximum or if the base manifold is compact.
\end{proof}

\bigskip
Under the hypothesis of the above theorem, the Ricci scalar function of the fiber cannot be negative.
We conclude:

\vspace{0.15in}

\begin{corollary}\label{rn}
Let $M=M_1^{n_1}\times_{f} M_2^{n_2}$ be an Einstein Finslerian warped product manifold of positive Ricci scalar function with base, a compact Riemannian manifold, and fiber a weakly Berwald Finsler manifold. Then in order that the warped product $M$ is a legitimate  Einstein Finsler warped product manifold, Ricci scalar function of the fiber must be positive.
\end{corollary}
 
\bigskip
 
 Dan Dumitru delves into the specific conditions related to the scalar curvature of the base and fiber of Riemannian warped product manifolds. These conditions lead to the elegant revelation that the resulting manifold is, in fact, a product of Riemannian manifolds, as demonstrated in Theorem $2.5$ of \cite{dam}. We extend this result to the  Finslerian settings.
 In view of the above corollary, it suffices to explore obstructions to the triviality of Finslerian warped product manifolds when $\operatorname{Scal}_M$ and $\operatorname{Scal}_{M_2}$ are strictly positive.

\bigskip

\begin{theorem}\label{thm4.2}
	Let $M=M_1^{n_1}\times_{f} M_2^{n_2}$ be an Einstein warped product Finsler manifold with 
 positive scalar curvature. 
 Let the base be a Riemannian manifold and fiber be  a weakly Berwald Finsler manifold with 
positive scalar curvature.
Suppose that 
  any of  the following conditions holds, then the 
  warping function is either subharmonic or superharmonic. Consequently, if the base manifold is compact 
  or if either $f$ attains its maximum or minimum, then
  it is a constant function. In this case,
   the warped product $M=M_1\times_f M_2$ is a weakly Berwald Finsler manifold, which is  simply a Riemannian product of two Finslerian manifolds.
	\begin{itemize}
		\item[\emph{(a)}] There exist a point $\Tilde{x_2}\in M_2$ such that 
  $\|\nabla f(x_1)\|\geq \sqrt{\frac{\operatorname{Scal}_{M_2}(\Tilde{x_2})}{n_2(n_2-1)}}$ for every $x_1$,
		
		\item[\emph{(b)}] $\operatorname{Scal}_M\leq \operatorname{Scal}_{M_1}$,
		
		\item[\emph{(c)}] $\operatorname{Scal}_M-\operatorname{Scal}_{M_2}\geq \operatorname{Scal}_{M_1}$ and $\frac{\operatorname{Scal}_{M_2}(x_2)}{\operatorname{Scal}_{M}(x_1,x_2)}\geq\frac{n_2}{(n_1+n_2)}$, for all $(x_1,x_2)\in M_1\times M_2$,
		
		\item[\emph{(d)}] $\operatorname{Scal}_M-\dfrac{\operatorname{Scal}_{M_2}}{f^2}\geq \operatorname{Scal}_{M_1}$, 
		
		\item[\emph{(e)}] $\operatorname{Scal}_{M_1}\leq 0$,
		
		\item[\emph{(f)}] $\operatorname{Scal}_{M_1}\geq\dfrac{n_1\operatorname{Scal}_{M_2}(x_2)}{n_2f^2(x_1)}$. 
	\end{itemize} 
\end{theorem}	

\begin{proof}
	{(a)} For $(x_1, x_2) \in M$ availing \eqref{f10},
	\begin{equation*}
		-f(x_1)\Delta f(x_1) + (n_2-1)\|\nabla f(x_1)\|^2+ \lambda(x_1,x_2)f^2(x_1)=\mu (x_2). 
	\end{equation*}
Then for $x_2=\Tilde{x_2}$, the above equation implies,
     \begin{equation*}
     	-f(x_1)\Delta f(x_1) + \lambda(x_1,\Tilde{x_2})f^2(x_1)=\mu (\Tilde{x_2})-(n_2-1)\|\nabla f(x_1)\|^2\leq 0,
     \end{equation*} 
 by using the hypothesis (a) which can be written as $\|\nabla f(x_1)\|\geq \sqrt{\frac{\mu(\Tilde{x_2})}{n_2-1}}$. This yields,
 \begin{equation*}
 	f(x_1)\Delta f(x_1)\geq \lambda(x_1,\Tilde{x_2})f^2(x_1).
 \end{equation*}
 Hence $\Delta f>0$ so warping function $f$ is constant, as base manifold is compact. \\

 { (b)} By \eqref{1.4},
 
 	$$\operatorname{Scal}_{M_1}= n_1\lambda(x_1,x_2)-n_2\frac{\Delta f(x_1)}{f(x_1)},$$
  equivalently
 	$$ \operatorname{Scal}_{M_1}+n_2\lambda(x_1,x_2)=(n_1+n_2)\lambda(x_1,x_2)-n_2\frac{\Delta f(x_1)}{f(x_1)}.$$
Consequently,
 	$$\operatorname{Scal}_{M}-\operatorname{Scal}_{M_1}=n_2\left(\lambda(x_1,x_2)+\frac{\Delta f(x_1)}{f(x_1)}\right),$$
 
which implies
$$\dfrac{\Delta f(x_1)}{f}+ \lambda(x_1,x_2)\leq 0,$$ 
by using assumption (b).
Hence $\Delta f \leq 0$, so warping function $f$ is constant.\newline
 
 {(c)} By \eqref{1.5},
 \begin{align*}
 	\operatorname{Scal}_{M_2}=n_2\mu(x_2) \geq n_2 \lambda(x_1,x_2),
 \end{align*}
by availing the part of the supposition $(c)$ which is $\mu(x_2)\geq\lambda(x_1,x_2)$ for all $(x_1,x_2)\in M_1\times M_2$.
 
Now by using \eqref{1.4} and the above equation,

	$$\operatorname{Scal}_{M_1}+ \operatorname{Scal}_{M_2}\geq (n_1+n_2)\lambda(x_1,x_2)-n_2 \frac{\Delta f(x_1)}{f(x_1)},$$
 thus,
	$$\operatorname{Scal}_{M_1}+ \operatorname{Scal}_{M_2}\geq \operatorname{Scal}_M-n_2 \frac{\Delta f(x_1)}{f(x_1)},$$
 which gives,   
 $$n_2 \frac{\Delta f(x_1)}{f(x_1)}	\geq \operatorname{Scal}_M -(\operatorname{Scal}_{M_1}+\operatorname{Scal}_{M_2})\geq 0,$$
by hypothesis in $(c)$. 
Hence, $\Delta f\geq0$ so warping function $f$ is constant.\newline
 
 {(d)} By \eqref{1.6}, we have
 
 	$$\operatorname{Scal}_{M_1}+\dfrac{\operatorname{Scal}_{M_2}}{f(x_1)^2}=\operatorname{Scal}_M+ n_2(n_2 -1)\dfrac{\|\nabla f\|^2}{f(x_1)^2}-2n_2\frac{\Delta f(x_1)}{f(x_1)}.$$
  
  Hence,
  
 	$$\operatorname{Scal}_M-\left(\operatorname{Scal}_{M_1}+\dfrac{\operatorname{Scal}_{M_2}}{f(x_1)^2}\right)=2n_2\frac{\Delta f(x_1)}{f(x_1)}-n_2(n_2 -1)\dfrac{\|\nabla f\|^2}{f(x_1)^2}\geq 0,$$
  
 which employing hypothesis in (d), 
  
 	$$2n_2\frac{\Delta f(x_1)}{f(x_1)}\geq n_2(n_2 -1)\dfrac{\|\nabla f\|^2}{f(x_1)^2}\geq 0.$$ 

This yields, $\Delta f\geq 0$ so warping function $f$ is constant.\newline

{(e)} By \eqref{1.4}, 
\begin{align*}
	n_2\dfrac{\Delta f(x_1)}{f(x_1)}&= n_1\lambda(x_1,x_2)-\operatorname{Scal}_{M_1}\geq 0,
\end{align*}
by hypothesis in $(e)$.
That is, $\Delta f\geq 0$ so warping function $f$ is constant.\newline

{(f)}
By \eqref{1.4} we have,

		$$\operatorname{Scal}_{M_1} = n_1\lambda(x_1,x_2)- n_2\dfrac{\Delta f(x_1)}{f(x_1)},$$

  thus,
		$$\operatorname{Scal}_{M_1} f^2(x_1) = n_1\lambda(x_1,x_2)f^2(x_1)- n_2f(x_1)\Delta f(x_1).\\$$

Now by using \eqref{f10},
$$\operatorname{Scal}_{M_1}f^2(x_1)-n_1\mu(x_2)=(n_1-n_2)f(x_1)\Delta f(x_1)-n_1(n_2-1)\|\nabla f(x_1)\|^2\geq 0,$$
by hypothesis in $(f)$.
Finally, we affirm that 
$$(n_1-n_2)f(x_1)\Delta f(x_1) \geq n_1(n_2-1)\|\nabla f(x_1)\|^2\geq 0.$$

Here the following three cases occur.\\
{ (i)} $n_1 = n_2$, then clearly 
 $\|\nabla f(x_1)\| = 0,$  and the warping function $f$ must be constant
 as base manifold is connected. \newline
{ (ii)} $n_1 \geq n_2$, then
clearly, $\Delta f \geq  0,$ so warping function $f$ is constant.\newline
{ (iii)} $n_1 \leq n_2$, then 
 $\Delta f \leq 0 $ consequently, warping function $f$ is constant.
\end{proof}

\vspace{0.15in}

\noindent
{\bf \large Final Conclusions:} 
In summary, for the base $(M_1^{n_1},g_1)$  Riemannian and the fiber $(M_2^{n_2},F_2)$ a weakly Berwald Finsler manifold, 
 whether Einstein warped product $(M,F)$ is trivial or not
is depicted in the following tables.
Thus,
we have answered the Besse's question  partially in the Finslerian settings as follows. \\

(i) The base $(M_1^{n_1},g_1)$
is a complete, Riemannian manifold
of finite volume with warping function bounded.\\

\begin{tabular}{ |p{3.5cm}|p{3.5cm}|p{3.5cm}|p{3.5cm}|  }
 \hline
 Dimension of Warped Product&  Ricci scalar Function $\lambda$ of $M$ &  Ricci Scalar Function $\mu$ of $M_2$ & Triviality of $M$\\
 \hline
 Any & Non-positive Constant & Any Constant& Trivial\\
 \hline
 $n\geq 3$ (Randers)& Non-positive Scalar Function & Any Scalar Function & Trivial\\
 \hline
\end{tabular}

\vspace{0.15in}

 (ii) The base $(M_1^{n_1},g_1)$
is a compact Riemannian manifold.\\

\begin{tabular}{ |p{3.5cm}|p{3.5cm}|p{3.5cm}|p{3.5cm}|  }
 \hline
 Dimension of Warped product& Ricci Scalar Function $\lambda$ of $M$ & Ricci Scalar Function $\mu$ of $M_2$ & Triviality of $M$\\
 \hline
 \hline
 $1-3$   & Positive (or Negative) Scalar Function    &Any Scalar Function&   Trivial\\
 \hline
 $n_1,n_2\geq 2$& Positive Scalar Function & Any Scalar Function & Nontrivial\\
 \hline
 Any&   Non-negative Scalar Function  & Non-positive Scalar Function at One Point   &Trivial\\
\hline
 Any &Non-negative Scalar Function & Non-negative Scalar Function&  If (a)-(f) Holds of Theorem \ref{thm4.2}, then Trivial\\
 
 \hline
\end{tabular}

\vspace{0.4in}

\bmhead{Acknowledgements}

The authors gratefully acknowledge Professor Bankteshwar Tiwari for many helpful discussions. Mohammad Aqib gratefully acknowledges  Harish-Chandra Research Institute for its doctoral fellowship. 

\section*{Declarations}

\begin{itemize}
\item Funding: Not applicable.

\item Conflict of interest/Competing interests: The authors have no conflict of interest and no financial interests for this article.

\item Ethics approval: The submitted work is original and not submitted to more than one journal for simultaneous consideration.

\item Consent to participate: Not applicable.

\item Consent for publication: Not applicable.

\item Availability of data and materials: This manuscript has no associated data.

\item Code availability: Not applicable.

\item Authors' contributions: Conceptualization, methodology, investigation, validation, writing-original draft, review, editing, and reading have been performed by all the authors of the paper.

\end{itemize}

\end{document}